\documentclass[10pt]{article}

\usepackage{amsfonts,amsthm}
\usepackage{amssymb}
\usepackage{amsmath}

\usepackage{tikz-cd}

\usepackage{mathabx}

\usepackage{sseq}
\usepackage{rotating}
\usepackage{bbm}
\usepackage[all,cmtip]{xy}
\usepackage{amscd}

\numberwithin{equation}{section}

\usepackage{caption}
\usepackage{subcaption}

\usepackage{tikz}
\usepackage{enumitem}

\usepackage{mathrsfs}
\usepackage{color}
\usepackage{xcolor}

\usepackage[latin1]{inputenc}
\usepackage{graphicx}
\usepackage{dsfont}

\usepackage[colorlinks]{hyperref}

\newtheorem{theorem}{Theorem}[section]

\newtheorem{lemma}[theorem]{Lemma}
\newtheorem{proposition}[theorem]{Proposition}

\theoremstyle{remark}
\newtheorem{remark}[theorem]{Remark}

\theoremstyle{definition}

\theoremstyle{theorem}
\newtheorem{TheoremA}{Theorem}

\newtheorem*{maintheorem}{Main Theorem}
\newtheorem*{Question}{Question}

\newcommand{\ow}{\omega}
\newcommand{\p}{\partial}

\newcommand{\C}{{\mathbb{C}}}

\newcommand{\R}{{\mathbb{R}}}

\newcommand{\Z}{{\mathbb{Z}}}
\newcommand{\N}{{\mathbb{N}}}

\renewcommand{\epsilon}{\varepsilon}
\renewcommand{\theta}{\vartheta}

\DeclareMathOperator{\ev}{ev}
\DeclareMathOperator{\reg}{reg}
\DeclareMathOperator{\crit}{crit}
\DeclareMathOperator{\End}{End}

\DeclareMathOperator{\dvol}{dvol}
\DeclareMathOperator{\Inj}{Inj}
\DeclareMathOperator{\Int}{int}

\DeclareMathOperator{\virdim}{vir-dim}

\DeclareMathOperator{\pt}{pt}
\DeclareMathOperator{\Cl}{Clif}
\DeclareMathOperator{\Ch}{Chek}

\DeclareMathOperator{\id}{id}
\DeclareMathOperator{\im}{Im}
\DeclareMathOperator{\ind}{ind}

\DeclareMathOperator{\Fix}{Fix}
\DeclareMathOperator{\Aut}{Aut}

\usepackage{graphicx}
\usepackage{authblk}

\begin{document}
\title{{
\bf
\Large
The Chekanov torus in $S^2\times S^2$ is not real
}}

\makeatletter
\newcommand{\subjclass}[2][2010]{%
  \let\@oldtitle\@title%
  \gdef\@title{\@oldtitle\footnotetext{#1 \emph{Mathematics Subject Classification.} #2.}}%
}
\newcommand{\keywords}[1]{%
  \let\@@oldtitle\@title%
  \gdef\@title{\@@oldtitle\footnotetext{\emph{Key words and phrases.} #1.}}%
}
\newcommand{\Date}[1]{%
  \let\@@@oldtitle\@title%
\gdef\@title{\@@@oldtitle\footnotetext{\emph{Date.} #1.}}%
}
\makeatother

\author
{
Joontae Kim
}
\date{}

\setcounter{tocdepth}{2}

\maketitle

\begin{abstract}
We prove that the count of Maslov index 2 $J$-holomorphic discs passing through a generic point of a real Lagrangian submanifold with minimal Maslov number at least two in a closed spherically monotone symplectic manifold must be even. As a corollary, we exhibit a genuine real symplectic phenomenon in terms of involutions, namely that the Chekanov torus $\mathbb{T}_{\Ch}$ in $S^2\times S^2$, which is a monotone Lagrangian torus not Hamiltonian isotopic to the Clifford torus $\mathbb{T}_{\Cl}$, can be seen as the fixed point set of a \emph{smooth} involution, but not of an \emph{antisymplectic} involution.
\end{abstract}

\section{Introduction and main result}
An even dimensional smooth manifold $M$ equipped with a closed non-degenerate 2-form $\ow$ is called a \emph{symplectic} manifold. An important class of submanifolds in a symplectic manifold $(M,\ow)$ are \emph{Lagrangian} submanifolds, namely those of middle dimension along which the symplectic form $\ow$ vanishes.

In symplectic topology, the study of Lagrangian submanifolds is a central topic; in particular, monotone Lagrangians exhibit rigidity phenomena. A Lagrangian $L$ in $M$ is called \emph{monotone} if there exists $K>0$ such that for all $\beta\in \pi_2(M,L)$ we have $\mu(\beta)=K\cdot \ow(\beta)$, where $\mu\colon \pi_2(M,L)\to \Z$ denotes the Maslov class of $L$. The \emph{minimal Maslov number} $N_L$ of a Lagrangian $L$ is defined as the non-negative integer $N_L$ satisfying $\mu\big(\pi_2(M,L)\big)=N_L\cdot \Z$.
A~typical example of a monotone Lagrangian is an embedded loop in $S^2$ dividing the sphere into two discs of equal area. A symplectic manifold $(M,\ow)$ containing a monotone Lagrangian is necessarily \emph{spherically monotone}, namely there exists $C>0$ such that for all $\alpha\in \pi_2(M)$ we have $c_1(M)[\alpha]=C\cdot \ow(\alpha)$.

We are interested in a special class of Lagrangians, called \emph{real Lagran\-gians}. A smooth involution $R$ on a symplectic manifold $(M,\ow)$ is called \emph{antisymplectic} if it satisfies $R^*\ow=-\ow$. Its fixed point set $\Fix(R)=\{x\in M\mid R(x)=x\}$ is a Lagrangian submanifold if it is nonempty. A \emph{real} Lagrangian in a symplectic manifold is a Lagrangian that is the fixed point set of an antisymplectic involution. Any real Lagrangian in a spherically monotone symplectic manifold is monotone. Historically, real Lagrangians first appeared in \cite[Theorem~C]{Givental} in the framework of the \emph{Arnold--Givental conjecture}.

In this paper, we study the following question in real symplectic topology.
\begin{Question}
Is there a monotone Lagrangian submanifold in a closed symplectic manifold that is the fixed point set of a \emph{smooth} involution, while not of an \emph{antisymplectic} involution?
\end{Question}
If we drop the condition that the Lagrangian is monotone, then a  favorite example of symplectic topologists, namely, the sphere $S^2$, immediately answers the above question. Indeed, consider any embedded loop $L$ as a Lagrangian in~$S^2$ which divides the sphere into two discs of \emph{different} areas. Since $L$ is not monotone, it is not real. Hence, $L$ cannot be the fixed point set of an \emph{antisymplectic} involution on $S^2$, while $L$ is the fixed point set of a \emph{smooth} involution, namely, the reflection interchanging the two discs. Taking the product with itself, we find a (non-monotone) Lagrangian in $S^2\times S^2$ that is the fixed point set of a smooth involution and that cannot be real.

To answer the above question in the monotone case, we need a \emph{symplectic} invariant that can serve as an obstruction to a monotone Lagrangian being real. For this purpose, we employ the mod 2 count of Maslov index~2 $J$-holomorphic discs passing through a generic point in a monotone Lagrangian $L$, which is denoted by $n(L)\in \Z_2$.  
This invariant was first used by Chekanov \cite[Section~2]{ChekDisc} and Eliashberg--Polterovich \cite[Proposition~4.1.A]{EliashPolt} based on Gromov's $J$-holomorphic discs \cite{Gromov}. This invariant is defined when the minimal Maslov number is at least 2, see Section \ref{sec: disccount} for the precise definition. In Section \ref{sec: symmetriccount}, we show that the count of such discs must be even when the Lagrangian is real, and hence we obtain the following main result of this paper.
\begin{maintheorem}
Any real Lagrangian $L$ with minimal Maslov number $N_L\ge 2$ in a closed spherically monotone symplectic manifold satisfies $n(L)=0$.
\end{maintheorem}
As an immediate corollary, we obtain a \emph{symplectic obstruction} for being a real Lagrangian; if a monotone Lagrangian $L$ satisfies $n(L)=1$, then it is not real. It is worth mentioning that the main theorem also holds for \emph{any connected component} of the real Lagrangian.

To find an example answering the question above, consider $S^2\times S^2$ equipped with the symplectic form $\ow\oplus \ow$, where $\ow$ denotes a Euclidean area form on $S^2$. The \emph{Clifford torus} $\mathbb{T}_{\Cl}=S^1\times S^1$ in $S^2\times S^2$ is the real Lagrangian torus defined as the product of the equators in each $S^2$-factor. The \emph{Chekanov torus} $\mathbb{T}_{\Ch}$ in~$S^2\times S^2$ is a monotone Lagrangian torus with minimal Maslov number $N_{\mathbb{T}_{\Ch}}=2$, which is not Hamiltonian isotopic but Lagrangian isotopic to the Clifford torus $\mathbb{T}_{\Cl}$. There are several equivalent ways to define the Chekanov torus $\mathbb{T}_{\Ch}$ in $S^2\times S^2$ up to Hamiltonian isotopy, see \cite{AF}, \cite[Theorem~1.1]{FOOO}, \cite[Section~3]{ChekanovSchlenk}, \cite[Section~4.2]{Gad}, and \cite[Example~1.22]{EP}. We refer to \cite[Theorem~1.1]{Usher} for the discussion of these descriptions and \cite[Section~3]{Chekanov} for the original construction of the Chekanov torus in $\R^{2n}$. 

It is known that $n(\mathbb{T_{\Ch}})=1$, see \cite[Lemmata~5.2 and 5.3]{ChekanovSchlenk}. 
Since $\mathbb{T}_{\Ch}$ is Lagrangian isotopic to the Clifford torus $\mathbb{T}_{\Cl}$, it is the fixed point set of a smooth involution, see Lemma \ref{lem: smoothinvolution}. As an application of the main theorem, it follows that the answer to the above question is yes. This is a genuine \emph{real symplectic phenomenon} in terms of \emph{involutions}.
\begin{TheoremA}\label{thm: thmA}
The Chekanov torus $\mathbb{T}_{\Ch}$ in $S^2\times S^2$ is the fixed point set of a smooth involution but not of an antisymplectic involution.
\end{TheoremA}
This result gives rise to another question. Every (not necessarily real) Lagrangian torus in $S^2\times S^2$ is Lagrangian isotopic to $\mathbb{T}_{\Cl}$, see \cite[Theorem~A]{DGI}. Since the Clifford torus $\mathbb{T}_{\Cl}$ is real, it is natural to ask if exotic monotone Lagrangian tori in $S^2\times S^2$ are real. Here \emph{exotic} means that it is not Hamiltonian isotopic to $\mathbb{T}_{\Cl}$. Theorem \ref{thm: thmA} tells us that the \emph{first} exotic Lagrangian torus, namely, the Chekanov torus, is not real. It is now tempting to conjecture that any real Lagrangian torus in $S^2\times S^2$ is Hamiltonian isotopic to the Clifford torus~$\mathbb{T}_{\Cl}$. This conjecture was recently proved in \cite{Kim2}.
We refer to \cite[Theorem~A]{Kim2} and \cite[Proposition~B]{Kim} for information on the classification of real Lagrangian submanifolds in $S^2\times S^2$.

We briefly outline the proof of the main theorem. Suppose that $L=\Fix(R)$ is a real Lagrangian. Any $\ow$-compatible almost complex structure~$J$ on $M$ which is anti-$R$-invariant together with the antisymplectic involution $R$ defines an involution $\mathcal{R}_0$ on the moduli space of Maslov index 2 $J$-holomorphic discs with boundary on $L$. In Section \ref{sec: nofixedpointonmoduli}, we show that the involution $\mathcal{R}_0$ has no fixed points, which implies that the count of the discs must be even. The count has to be done for an anti-$R$-invariant $J$ which is \emph{regular} to guarantee that the moduli space is \emph{transversely} cut out. This is not obvious, but since $\mathcal{R}_0$ has no fixed points, we achieve \emph{equivariant transversality} in Section \ref{sec: eqtrans}, following ideas of Khovanov--Seidel \cite[Section~5c]{SeidelQuiv}. See also \cite[Section~14c]{SeidelBook}.

Recently, Brendel \cite{Brendel} found a different symplectic obstruction for being a real Lagrangian in \emph{almost toric} symplectic manifolds, using the so-called \emph{versal deformations} developed in \cite{Chekanov}. This invariant is less technical and easily computable. Nevertheless, the symplectic obstruction in terms of $J$-holomorphic discs is still interesting in its own right due to its relation with open Gromov--Witten invariants, and it is applicable to a wider class of symplectic manifolds.

\subsection*{Organization of the paper}
In Section \ref{sec: disccount}, we define the mod 2 count of Maslov index 2 $J$-holomorphic discs, as our main tool. In Section \ref{sec: symmetriccount}, we establish results that will be needed to prove the main theorem. In Section \ref{sec: proofofmainthm}, we prove the main theorem and Theorem \ref{thm: thmA}. Finally, we provide an appendix in Section \ref{sec: appendix}, containing (i) some properties of $J$-holomorphic discs that are mainly used to prove Lemma \ref{lem: inj_Rdenseopen}, (ii) a simple proof of the main theorem for $S^2\times S^2$, and (iii) an alternative proof of equivariant transversality using the doubling construction of Hofer--Lizan--Sikorav.

\section{Mod 2 count of Maslov index 2 $J$-holomorphic discs}\label{sec: disccount}
We give a quick review on the mod 2 count of Maslov index 2 $J$-holomorphic discs passing through a generic point in a Lagrangian $L$. For more details we refer to \cite[Section~2]{ChekDisc}  and also to the more recent expositions \cite[Section~3.1]{Auroux} and \cite[Section~6.1]{Vian2}.

Let $(M,\ow)$ be a closed spherically monotone symplectic manifold and let~$L$ be a monotone Lagrangian in $M$ with minimal Maslov number $N_L=2$. We abbreviate by $\mathcal{J}$ the \emph{space of $\ow$-compatible almost complex structures on~$M$} endowed with the $C^\infty$-topology. Fix $J\in \mathcal{J}$. We let
\begin{equation}\label{eq: spaceofJdiscs}
\widehat{\mathcal{M}}(L,J)=\big\{u\colon (D^2,\p D^2)\to (M,L) \mid \bar{\p}_J(u)=0,\ \mu(u)=2\big\}	
\end{equation}
be the \emph{space of Maslov index 2 $J$-holomorphic discs with boundary on $L$}. Here $\bar{\p}_J(u)=\frac{1}{2}(du+J\circ du\circ i)$ denotes the complex antilinear part of $du$. We abbreviate by
$$
\mathcal{M}_1(L,J) = \widehat{\mathcal{M}}(L,J)\big/\Aut(D^2,1)
$$
the \emph{moduli space of Maslov index 2 $J$-holomorphic discs with boundary on~$L$ with one boundary marked point}, where $\Aut(D^2,1)$ is the group of biholomorphisms of the closed unit disc $D^2\subset \C$ (smooth up to the boundary) fixing $1\in \p D^2$.

A $J$-holomorphic disc $u\colon (D^2,\p D^2)\to (M,L)$ is called \emph{simple} if the \emph{set of injective points of $u$}
$$
\Inj(u):=\big\{z\in D^2\mid du(z)\ne 0,\ u^{-1}(u(z))=\{z\}\big\}
$$
is dense in $D^2$. This set is always open, see for instance \cite[Proposition~2.2]{ZehAnnulus}. A result of Lazzarini \cite[Theorem~A]{Laz} implies that if the homotopy class of a $J$-holomorphic disc $u\colon (D^2,\p D^2)\to (M,L)$ in $\pi_2(M,L)$ is indecomposable (by means of simple $J$-holomorphic discs), then $u$ is simple. See also \cite[Lemma~6]{GeigesKai} where the same argument is used. Since $N_L>0$ it follows that every $J$-holomorphic disc $u$ with $\mu(u)=N_L$ is simple. 

The classical transversality theorem \cite[Section~3.2]{McSalJholo} asserts that for generic $J\in \mathcal{J}$ the moduli space $\mathcal{M}_1(L,J)$ is a smooth manifold of dimension~$n$. Indeed, such a $J$ is \emph{regular}, see Section \ref{sec: eqtrans} for the definition. Since every $u\in \widehat{\mathcal{M}}(L,J)$ satisfies $\mu(u)=2$ which is the smallest possible positive Maslov index and since $L$ is monotone, no bubbling of discs or spheres can occur a priori. By Gromov's compactness theorem \cite{UrsGromov}, $\mathcal{M}_1(L,J)$ is a compact smooth manifold \emph{without} boundary.

For regular $J\in \mathcal{J}$ the mod 2 Brouwer degree of the  evaluation map
$$
\ev\colon \mathcal{M}_1(L,J) \to L,\quad \ev[u]=u(1)
$$
is called the \emph{mod 2 count of Maslov index 2 $J$-holomorphic discs} passing through a generic point in $L$, and is denoted by
$$
n(L,J):=\deg_2(\ev)\in \Z_2.
$$
By a standard cobordism argument as in \cite[Section~2.2]{ChekDisc} or \cite[Theorem~6.6.1]{McSalJholo}, the value $n(L,J)\in \Z_2$ is independent of the choice of regular~$J$. In the sequel we suppress $J$ in the notation. Note that $n(L)\in \Z_2$ is given by the mod 2 cardinality of the finite set $\ev^{-1}(x)$ for any regular value $x\in L$ of~$\ev$, see \cite[p.~24]{Milnor}. By convention, if $\mathcal{M}_1(L,J)$ is empty (for example, when $N_L\ge 3$) we set $n(L)=0$.

\begin{remark}
It is worth noting that this invariant can be \emph{enumerative}, i.e., $n(L)$ takes values in $\Z$, if the Lagrangian $L$ is oriented and spin. A spin structure determines an orientation of $\mathcal{M}_1(L,J)$ and hence we can use the integer degree of the evaluation map. See \cite[Section~5]{Cho} and \cite[Section~5.5]{Vian2}. 
Moreover, $n(L)$ is invariant under symplectomorphisms of $L$, i.e., $n(L)=n(\phi(L))$ for any symplectomorphism $\phi$, although this fact is not used in the paper.
\end{remark}

\section{Symmetric count for real Lagrangians}\label{sec: symmetriccount}
We exhibit that the count happens \emph{symmetrically} when the Lagrangian is real. With the notation from Section \ref{sec: disccount}, we assume that the Lagrangian $L$ is real. Hence, we write $L=\Fix(R)$ for some antisymplectic involution $R$ on~$M$.
\subsection{Involution on moduli spaces}\label{sec: involutiononmoduli}
Fix $J\in \mathcal{J}$ which is \emph{anti-$R$-invariant}, i.e., $J=-R^*J:=-R_*JR_*$. We write $\mathcal{J}_R$ for the \emph{space of anti-$R$-invariant $\ow$-compatible almost complex structures on~$M$}. Note that $\mathcal{J}_R$ is nonempty and contractible, see \cite[Proposition~1.1]{Welsch}. For a technical reason, we deal with the \emph{moduli space of Maslov index 2 $J$-holomorphic discs with boundary on $L$},
$$
\mathcal{M}_0(L,J) = \widehat{\mathcal{M}}(L,J)\big/\Aut(D^2),
$$
where $\Aut(D^2)$ is the group of biholomorphisms of the unit disc $D^2\subset \C$. Since $J\in \mathcal{J}_R$, we can consider the involution
\begin{equation}\label{eq: involutiononmoduli}
\widehat{\mathcal{R}}\colon \widehat{\mathcal{M}}(L,J)\to \widehat{\mathcal{M}}(L,J),\quad u\mapsto R\circ u \circ \rho,	
\end{equation}
where $\rho\colon D^2\to D^2$ denotes complex conjugation on $D^2\subset \C$. 
\begin{lemma}
The involution $\mathcal{R}_j$ on $\mathcal{M}_j(L,J)$ defined by
$$
\mathcal{R}_j[u]:=\big[\widehat{\mathcal{R}}(u)\big]=[R\circ u\circ \rho]
$$
is well-defined for $j=0,1$.
\end{lemma}
\begin{proof}
We prove this only for $\mathcal{R}_1$ as the case of $\mathcal{R}_0$ follows easily.
	Recall that every $\sigma\in \Aut(D^2,1)$ is of the form
	$$
	\sigma(z)=e^{i\theta}\frac{z-z_0}{1-\bar{z}_0z},
	$$
for some $\theta\in \R$ and $|z_0|<1$ satisfying  $e^{i\theta}\frac{1-z_0}{1-\bar{z}_0}=1$, see \cite[Exercise~4.2.5]{McSalJholo}. The map $\tilde{\sigma}(z)=e^{-i\theta}\frac{z-\bar{z}_0}{1-z_0z}\in \Aut(D^2,1)$ satisfies $\sigma\circ \rho=\rho\circ \tilde{\sigma}$. It follows that
$$
\mathcal{R}_1[u\circ \sigma]=[R\circ u\circ \sigma \circ \rho] = [R\circ u\circ \rho \circ \tilde{\sigma}]=\mathcal{R}_1[u].
$$
This shows that $\mathcal{R}_1$ is well-defined.
\end{proof}
\subsection{No fixed point of the involutions $\mathcal{R}_0$ and $\mathcal{R}_1$}\label{sec: nofixedpointonmoduli}
We start with observing that for $u\in \widehat{\mathcal{M}}(L,J)$,
$$
\mathcal{R}_0[u]=[u] \iff R\circ u\circ \rho\circ \sigma =u \quad \text{for some $\sigma\in \Aut(D^2)$}.
$$
This section is devoted to prove the following.
\begin{theorem}\label{thm: nofixedthm}
Let $J\in \mathcal{J}_R$. For any $u\in\widehat{\mathcal{M}}(L,J)$ there exists no  $\sigma\in \Aut(D^2)$ such that
$$
R\circ u \circ \rho \circ \sigma = u.
$$
In particular, the fixed point set of $\mathcal{R}_j$ is empty for $j=0,1$.
\end{theorem}
\begin{remark}
While the fact that $\Fix(\mathcal{R}_1)=\emptyset$ will be used to show that the count of discs is even, the fact that $\Fix(\mathcal{R}_0)=\emptyset$ is needed to prove equivariant transversality.
\end{remark}
The idea of the proof is the following. If $u$ satisfies $R\circ u \circ \rho \circ \sigma=u$ for some $\sigma \in \Aut(D^2)$, then we can decompose $[u]$ into two $J$-holomorphic discs $[u_1]$ and $[u_2]$ in $\pi_2(M,L)$ with $\mu[u_j]\ge 2$. This is impossible since then $\mu[u]=\mu[u_1]+\mu[u_2]\ge 4$.
We first need the following lemma. 
\begin{lemma}\label{lem: rhoisinvolution}
Let $u\colon (D^2,\p D^2)\to (M,L)$ be a simple $J$-holomorphic disc. Suppose that there exists $\sigma\in \Aut(D^2)$ such that
$$
R\circ u \circ \rho \circ \sigma = u.
$$
Then $\tilde{\rho}:=\rho\circ \sigma$ is an involution on $D^2$ and $\Fix(\tilde{\rho})$ is a simple smooth arc with ends in $\p D^2$. 
\end{lemma}
\begin{proof}
Suppose that there exists $\sigma\in \Aut(D^2)$ such that
\begin{equation}\label{eq: fixpointrelation}
R\circ u \circ \rho \circ \sigma = u.	
\end{equation}
The diffeomorphism $\tilde{\rho}=\rho\circ \sigma\colon D^2\to D^2$ is antiholomorphic. Applying \eqref{eq: fixpointrelation} twice we obtain
\begin{equation}\label{eq: tilderho}
u\circ \tilde{\rho}^2=u.	
\end{equation}
For a given $z\in D^2$ choose a sequence of $z_\nu\in \Inj(u)$ converging to $z$. From \eqref{eq: tilderho} we have $\tilde{\rho}^2(z_\nu)=z_\nu$ for all $\nu\in \N$, and hence $\tilde{\rho}^2(z)=z$. This proves that $\tilde{\rho}$ is an involution on $D^2$. Since $\tilde{\rho}$ is an antiholomorphic involution of $D^2$, it is conjugated to complex conjugation $z\mapsto \bar{z}$ and hence $\Fix(\tilde{\rho})$ is a simple smooth arc with ends in $\p D^2$. See also \cite[Theorem~3.3]{Kerek} for a topological result.
\end{proof}
The \emph{energy} of a smooth disc $u\colon (D^2,\p D^2)\to (M,L)$ is defined by
$$
E(u)=\frac{1}{2}\int_{D^2}|du|^2\dvol,
$$
where $\dvol$ denotes the standard volume form on $D^2$ and the norm $|\cdot |$ is induced by the metric $g_J=\ow(\cdot,J\cdot)$.
If $u$ is $J$-holomorphic, then
$$
E(u)=\ow(u):=\int_{D^2}u^*\ow.
$$
We observe that $E(u)>0$ if and only if there exists  $z\in D^2$ such that $u(z)\in M\setminus L$. 
The next step is to show the following.
\begin{lemma}\label{lem: 2discsdecomp}
Let $u\colon (D^2,\p D^2)\to (M,L)$ be a simple $J$-holomorphic disc with positive energy $E(u)>0$. Suppose that there exists $\sigma\in \Aut(D^2)$ such that
\begin{equation}\label{eq: fixedpoint}
R\circ u \circ \rho \circ \sigma = u.	
\end{equation}
Then there exist $J$-holomorphic discs $u_j\colon (D^2,\p D^2)\to (M,L)$ with $E(u_j)>0$ for $j=1,2$, such that $[u]=[u_1]+[u_2]\in \pi_2(M,L)$ and $E(u_1)=E(u_2)$.
\end{lemma}
\begin{proof}
We again write $\tilde{\rho}=\rho\circ \sigma$. By Lemma \ref{lem: rhoisinvolution}, the map $\tilde{\rho}$ is an involution and $\Fix(\tilde{\rho})$ divides $D^2$ into two closed discs $D_1$ and $D_2$ with the following properties:
\begin{enumerate}
	\item[(i)] $D^2=D_1\cup D_2$,
	\item[(ii)] $\Fix(\tilde{\rho})=D_1\cap D_2$,
	\item[(iii)] $\tilde{\rho}(D_1)=D_2$.
\end{enumerate}
See the right drawing in Figure \ref{fig: decomposition}. The third property follows from the fact that $\tilde{\rho}$ is orientation-reversing. We shall construct two $J$-holomorphic discs $u_j\colon (D^2,\p D^2)\to (M,L)$ for $j=1,2$ with $E(u_j)>0$, which are extracted  from the regions $D_1$ and $D_2$.

Choose a biholomorphism $\phi_1$ from $\Int(D^2)$ to $\Int(D_1)$ that extends to a continuous map from $D^2$ to $D_1$, see for example \cite[Proposition~5.3]{LazGAFA}. By~\eqref{eq: fixedpoint}, we observe that $u(z)\in L$ for all $z\in \Fix(\tilde{\rho})$. Consider the continuous map
$$
u_1:=u\circ \phi_1\colon (D^2,\p D^2) \to (M,L),
$$
which is smooth and $J$-holomorphic on $\Int(D^2)$. Since $u$ has a finite energy, so does $u_1$. By removable singularities on the boundary \cite[Theorem~B.1]{LazGAFA}, $u_1$ extends to a smooth map on $D^2$ (still denoted by $u_1$) and hence is a $J$-holomorphic disc with boundary on $L$. Similarly, using the biholomorphism $\phi_2:=\tilde{\rho} \circ \phi_1\circ \tilde{\rho}$ from $\Int(D^2)$ to $\Int(D_2)$, we obtain the $J$-holomorphic disc corresponding to $D_2$,
$$
u_2\colon (D^2,\p D^2)\to (M,L)
$$
such that $u_2=u\circ \phi_2$ on $\Int(D^2)$. See Figure \ref{fig: decomposition}. By construction, we obtain $[u]=[u_1]+[u_2]\in \pi_2(M,L)$ and $E(u)=E(u_1)+E(u_2)$.
It remains to show that $E(u_1)=E(u_2)$. Since $\displaystyle E(u_j)=\int_{D^2}u_j^*\ow<\infty$ and $\displaystyle \int_{D^2}u_j^*\ow=\int_{\Int(D^2)}u_j^*\ow$ for $j=1,2$, it suffices to show that
$$
\int_{\Int(D^2)}u_1^*\ow = \int_{\Int(D^2)}u_2^*\ow.
$$
To see this, using \eqref{eq: fixedpoint} we observe that on $\Int(D^2)$,
$$
R\circ u_1 \circ \tilde{\rho} = R\circ u\circ \phi_1\circ \tilde{\rho} = u\circ \tilde{\rho} \circ \phi_1\circ \tilde{\rho} = u_2.
$$
Therefore, we verify that
\begin{align*}
 \int_{\Int(D^2)} u_2^* \ow &=\int_{\Int(D^2)} (R\circ u_1\circ \tilde{\rho})^*\ow \\
 &= \int_{\Int(D^2)} \tilde{\rho}^*u_1^*R^*\ow \\
 &= - \int_{\Int(D^2)} \tilde{\rho}^*u_1^*\ow \\
 &=  \int_{\Int(D^2)}u_1^*\ow. 
\end{align*}
The last equality follows from the fact that $\tilde{\rho}$ is orientation-reversing. This proves that $E(u_1)=E(u_2)$ and completes the proof.
\end{proof}
\begin{figure}[h]
\begin{center}
\begin{tikzpicture}[scale=0.5]

\begin{scope}[yshift=2.2cm]
\draw [fill=pink](-5,0) circle [radius=2];
\node at (-5,0) [above]{$D^2$};
\end{scope}

\begin{scope}[yshift=-2.2cm]
\draw [fill=blue!20](-5,0) circle [radius=2];
\node at (-5,0) [above]{$D^2$};
\end{scope}
\begin{scope}[yshift=0.8cm]	
\draw [->] (-1,1)--(1,0.5);
\end{scope}
\begin{scope}[yshift=-0.8cm]
\draw [->] (-1,-1)--(1,-0.5);	
\end{scope}


\begin{scope}[scale=1.4, xshift=-1cm]
\draw [fill=blue!20](5,0) circle [radius=2];

\draw[fill=pink, thick] ([shift=(20:2cm)]5,0) arc (20:190:2cm) -- plot [smooth] coordinates {(3.05,-0.4) (3.5,-0.9)  (4,-1.1) (4.2, -1) (4.5, -0.3) (5,0.3) (5.5, 1.3) (5.8, 1.3) (6, 0.4) (6.5, 0.1) (6.8,0.5)} -- cycle;

\draw [very thick, white](5,0) circle [radius=2];
\draw (5,0) circle [radius=1.97];

\draw [->] (6.5,2)--(6,1.1);
\node at (1,1.5) [above]{$\phi_1$};
\node at (1,-1.5) [below]{$\phi_2$};



\node at (4,0.2) [above]{$D_1$};
\node at (5.8,0) [below]{$D_2$};
\node at (6.5,2) [right]{$\Fix(\tilde{\rho})$};
	
\end{scope}

\end{tikzpicture}		
\end{center}
\caption{The description of the maps $\phi_j\colon D^2\to D_j$.}
\label{fig: decomposition}
\end{figure}
\begin{remark}
	In Lemmata \ref{lem: rhoisinvolution} and \ref{lem: 2discsdecomp}, we do not require that $L$ is monotone.
\end{remark}
We are ready to prove Theorem \ref{thm: nofixedthm}.
\begin{proof}[Proof of Theorem \ref{thm: nofixedthm}]
Suppose to the contrary that $R\circ u\circ \rho\circ \sigma=u$ for some $\sigma\in \Aut(D^2)$. Since $\mu(u)=2$, the disc $u$ is simple. By Lemma \ref{lem: 2discsdecomp}, there exists $J$-holomorphic discs $u_1$ and $u_2$ with $E(u_j)>0$ for $j=1,2$, such that $[u]=[u_1]+[u_2]\in \pi_2(M,L)$ and $E(u_1)=E(u_2)$. By the monotonicity of $L$, we have $\mu[u_1]=\mu[u_2]\ge 2$. Hence, we obtain
$$
\mu[u]=\mu[u_1]+\mu[u_2]\ge 4,
$$
which contradicts to the assumption that $\mu[u]=2$.
\end{proof}

\subsection{Equivariant transversality}\label{sec: eqtrans}
Since the involution $\mathcal{R}_0$ has no fixed point by Theorem \ref{thm: nofixedthm}, we achieve transversality by a local perturbation of an almost complex structure $J$ in~$\mathcal{J}_R$.

We explain relevant notations that will be used in this section. Fix $p>2$. For $J\in \mathcal{J}$ and $u\colon (D^2,\p D^2)\to (M,L)$ with $\bar{\p}_J(u)=0$, let
$$
{\bf D}_u\colon W^{1,p}(u^*TM,{u|_{\p D^2}^*}TL)\to L^p(u^*TM)
$$
be the Cauchy--Riemann type operator given by the linearization of the holomorphic curve equation $\bar{\p}_J(u)=0$, where
$W^{1,p}(u^*TM,{u|_{\p D^2}^*}TL)$ denotes the Banach space of $W^{1,p}$-sections of $u^*TM$ with Lagrangian boundary condition $u|_{\p D^2}^*TL$ and $L^p(u^*TM)$ denotes the Banach space of $L^p$-sections of $u^*TM$. By the Riemann--Roch theorem \cite[Theorem~C.1.10]{MS}, the operator ${\bf D}_u$ is Fredholm and its index is given by
$$
\ind({\bf D}_u)=n + \mu(u),
$$
where $\mu(u)$ is the Maslov index of the disc $u$. See \cite[Theorem~C.3.6]{McSalJholo} for the index computation. The \emph{virtual dimension} of the moduli space $\mathcal{M}_1(L,J)$ at~$[u]$ is given by
$$
\virdim_{[u]} \mathcal{M}_1(L,J)=n+\mu[u]-2=n.
$$
We say that $J\in \mathcal{J}$ is \emph{regular} if  the operator ${\bf D}_u$ is surjective for all $u\in \widehat{\mathcal{M}}(L,J)$. If $J$ is regular, then by the implicit function theorem $\mathcal{M}_1(L,J)$ is a smooth manifold and the (local) virtual dimension of $\mathcal{M}_1(L,J)$ is the dimension of $\mathcal{M}_1(L,J)$.
In the sprit of equivariantly regular points introduced in \cite[Section~4]{FHS} and \cite[Section~5c]{SeidelQuiv}, we denote by
\begin{align*}
\Inj_R(u) &= \big\{z\in \Inj(u) \mid u(z)\notin R(u(D^2)) \big\}	 \\
&= \big\{ z\in D^2 \mid du(z)\ne 0,\ u^{-1}(u(z))=\{z\},\ u(z)\notin R(u(D^2))\big\}
\end{align*}
the \emph{set of $R$-injective points} of $u$.  

Fix $J\in \mathcal{J}_R$. Recall that $\widehat{\mathcal{R}}(u)=R\circ u\circ \rho$ is defined in \eqref{eq: involutiononmoduli}. We begin with the following basic fact. 
\begin{lemma}\label{lem: simpleRsimple}
A $J$-holomorphic disc $u$ is simple if and only if $\widehat{\mathcal{R}}(u)$ is simple.
\end{lemma}
\begin{proof}
This follows from the fact that $R$ and $\rho$ are diffeomorphisms.
\end{proof}

The following lemma is the crucial ingredient to prove equivariant transversality. This is the analogue of \cite[Lemma~5.12]{SeidelQuiv}.
\begin{lemma}\label{lem: inj_Rdenseopen}
Let $u\in\widehat{\mathcal{M}}(L,J)$. Then $\Inj_R(u)$ is open and dense in $D^2$.
\end{lemma}
\begin{proof}
We first observe that 
$$
\Inj_R(u)=\Inj(u)\cap S\big(u,\widehat{\mathcal{R}}(u)\big),
$$
where $S\big(u,\widehat{\mathcal{R}}(u)\big)=\big\{z\in D^2\mid u(z)\notin R(u(D^2))\big\}$ is the open set defined in \eqref{eq: Su_1u_2set}. Hence, $\Inj_R(u)$ is open. It remains to prove that $\Inj_R(u)$ is dense in~$D^2$. To show this, it suffices to show that $S\big(u,\widehat{\mathcal{R}}(u)\big)$ is dense. Note that $\rho(D^2)=D^2$ and $\rho(\p D^2)=\p D^2$. By Theorem \ref{thm: nofixedthm}, we know $\mathcal{R}_0[u]\ne [u]$ in $\mathcal{M}_0(L,J)$, equivalently,
$$
R\circ u\circ \rho \circ \sigma \ne u\quad \text{for any $\sigma\in \Aut(D^2)$}.
$$
Since $u$ is simple, so is $\widehat{\mathcal{R}}(u)$ by Lemma \ref{lem: simpleRsimple}. By Proposition \ref{prop: imagesamerep}, we have $u(D^2)\ne R(u(D^2))$ or $u(\p D^2)\ne R(u(\p D^2))$. Since $u(\p D^2)\subset L=\Fix(R)$, we obtain
	$$
	R(u(\p D^2)) = u(\p D^2)
	$$
and therefore $u(D^2)\ne R(u(D^2))$. It now follows from Proposition \ref{prop: imagediffdiscrete} that $S\big(u,\widehat{\mathcal{R}}(u)\big)$ is dense.
\end{proof}
We now prove the equivariant transversality following arguments in \cite[Proposition~5.13]{SeidelQuiv}.
\begin{theorem}\label{thm: equitrans}
	There exists a Baire subset $\mathcal{J}_R^ {\reg}\subset \mathcal{J}_R$ which has the property that every $J\in \mathcal{J}_R^{\reg}$ is regular.
\end{theorem}
\begin{proof}
Fix a positive integer $\ell$. We abbreviate by
$$
\mathcal{J}_R^\ell=\{J\in \mathcal{J}_R \mid \text{$J$ is of class $C^\ell$}\}.
$$
Its tangent space $T_J\mathcal{J}_R^\ell$ at $J$ consists of $C^\ell$-sections $Y\colon M\to \End(TM)$ satisfying
\begin{enumerate}
	\item[(i)] $Y=Y^*=JYJ$, where $Y^*$ denotes the adjoint operator with respect to the metric $g_J=\ow(\cdot,J\cdot)$,
	\item[(ii)] $R^*Y=-Y$.
\end{enumerate}
For $u\in \widehat{\mathcal{M}}(L,J)$ and $J\in \mathcal{J}_R^\ell$ we consider the operator (given by the linearization of the holomorphic curve equation $\bar{\p}_J(u)=0$ regarding $(u,J)$ as variables)
\begin{align*}
	\widetilde{{\bf D}}_{u,J}\colon W^{1,p}(u^*TM,u|_{\p D^2}^*TL)\oplus T_J\mathcal{J}_R^\ell &\to L^p(u^*TM) \\
	(\xi, Y) &\mapsto {\bf D}_u\xi + Y(u)\p_tu.
\end{align*}
Since $\im({\bf D}_u)\subset \im(\widetilde{{\bf D}}_{u,J})$ and ${\bf D}_u$ has finite dimensional cokernel, so does the operator $\widetilde{{\bf D}}_{u,J}$. Hence, $\widetilde{{\bf D}}_u$ has closed image. To show that $\widetilde{{\bf D}}_{u,J}$ is surjective, it therefore suffices to prove that the image of $\widetilde{{\bf D}}_{u,J}$ is dense. Assume to the contrary that there exists a nonzero $Z\in L^q(u^*TM)$, with $ 1/p+1/q=1$, which satisfies
\begin{align}\label{eq: transintegral}
	&\int_{D^2}\ow({\bf D}_u\xi , J(u)Z)\dvol = 0,\quad \text{for all $\xi\in W^{1,p}(u^*TM,u|_{\p D^2}^*TL)$},\nonumber \\
	&\int_{D^2}\ow(Y(u)\p_tu,J(u)Z)\dvol = 0,\quad \text{for all $Y\in T_J\mathcal{J}_R^\ell$}.
\end{align}
The first equation implies that $Z$ is of class $W^{1,q}$ and ${\bf D}_u^*Z=0$, where ${\bf D}_u^*$ denotes the formal adjoint operator of ${\bf D}_u$. By elliptic regularity, $Z$ is of class~$C^\ell$.\\\\
{\bf Claim.} \emph{$Z(z)=0$ for all $z\in \Inj_R(u)$.}
\vspace{0.2cm}

\noindent
Assume, by contradiction, that $Z(z_0)\ne 0$ for some $z_0\in \Inj_R(u)$. By linear algebra (see \cite[Lemma~3.2.2]{McSalJholo}), we can choose $Y_{z_0}\in \End(T_{z_0}M)$ such that $Y_{z_0}=Y^*_{z_0}=J(u(z_0))Y_{z_0}J(u(z_0))$ and $Y_{z_0}\p_tu(z_0)=Z(z_0)$. Take any section $Y\in T_J\mathcal{J}^\ell$ such that $Y(u(z_0))=Y_{z_0}$. Choose a neighborhood $V\subset D^2$ of $z_0$ such that $\ow(Y(u) \p_tu,J(u)Z)>0$ on $V$. Since $u(D^2\setminus V)$ and $R(u(D^2))$ are compact and $u(z_0)\notin u(D^2\setminus V)\cup R(u(D^2))$ (this follows as $z_0 \in \Inj_R(u)$), there exists a neighborhood $U\subset M$ of $u(z_0)$ such that
$$
u(D^2\setminus V)\cap U = \emptyset = R(u(D^2)) \cap U.
$$
Using a bump function supported in $U$, we construct a nonequivariant $Y\in T_J\mathcal{J}^\ell$ such that $\ow(Y(u)\p_tu,J(u)Z)>0$ on a small neighborhood $V'\subset V$ of $z_0$ and vanishes outside this neighborhood, yielding that the integral \eqref{eq: transintegral} with this~$Y$ is strictly positive. One directly checks that $R^*Y:=R_*YR_*\in T_J\mathcal{J}^\ell$. We consider the averaged tangent vector $\tilde{Y}=Y-R^*Y\in T_J\mathcal{J}_R^\ell$, which is now equivariant, i.e., $R^*\tilde{Y}=-\tilde{Y}$. We claim that the integral \eqref{eq: transintegral} with $\tilde{Y}$ is still strictly positive, which completes the above claim. To see this, it suffices to show that the integral \eqref{eq: transintegral} with $R^*Y$ vanishes. Since $R(u(D^2))\cap U=\emptyset$, the vector $R^*Y$ vanishes on $D^2$.  
This implies that $R^*Y(u) \p_t u(z)=0$ for all $z\in D^2$, which proves the subclaim.  Hence, this proves the claim that $Z(z)=0$ for all $z\in \Inj_R(u)$. 

Since $\Inj_R(u)$ is dense in $D^2$, we deduce that $Z\equiv 0$ so that the operator $\widetilde{{\bf D}}_u$ is surjective. The remaining proof is fairly standard, i.e., applying the parametric transversality theorem and Taubes' trick. We refer to \cite[Section~3.2]{McSalJholo}.
\end{proof}

\section{Proofs of the main theorem and Theorem \ref{thm: thmA}}\label{sec: proofofmainthm}
We are in position to prove the main theorem.
\begin{proof}[Proof of Main Theorem]
By Theorem \ref{thm: equitrans}, we can choose $J\in \mathcal{J}_R^{\reg}$. Since the involution $\mathcal{R}_1$ has no fixed point by Theorem \ref{thm: nofixedthm}, the restriction $\mathcal{R}_1$ to $\ev^{-1}(x)$ has no fixed point either. In other words, the elements in $\ev^{-1}(x)$ must come in pairs. This proves the theorem.
\end{proof}
As we mentioned in the introduction, we show the following simple topological result.
\begin{lemma}\label{lem: smoothinvolution}
	Any Lagrangian torus in $S^2\times S^2$ is given by the fixed point set of a smooth involution.
\end{lemma}
\begin{proof}
	Let $L$ be a Lagrangian torus in $S^2\times S^2$. By \cite[Theorem~1.1]{DGI}, $L$~is Lagrangian isotopic to $\mathbb{T}_{\Cl}$, namely there is an isotopy of Lagrangian embeddings from $L$ to $\mathbb{T}_{\Cl}$. By the smooth isotopy extension theorem, we choose an ambient isotopy $\phi_t$, $t\in [0,1]$ of $S^2\times S^2$ from $L$ to $\mathbb{T}_{\Cl}$, i.e., $\phi_t$ are diffeomorphisms for all $t\in [0,1]$, $\phi_0=\id_{S^2\times S^2}$ and $\phi_1(L)=\mathbb{T}_{\Cl}$. We write $R_{\Cl}$ for the obvious antisymplectic involution on $S^2\times S^2$ with fixed point set $\mathbb{T}_{\Cl}=\Fix(R_{\Cl})$. For simplicity, we let $\phi:=\phi_1$. Consider the smooth involution on $S^2\times S^2$ given~by
	$$
	R_L:=\phi^{-1}\circ R_{\Cl}\circ \phi.
	$$
	We claim that $\Fix(R_L)=L$. This follows from the observation that
	\begin{eqnarray*}
		x \in \Fix(R_L) &\iff& R_{\Cl}\big(\phi(x)\big)= \phi(x) \\
		&\iff& \phi(x)\in \mathbb{T}_{\Cl} \\
		&\iff& x\in \phi^{-1}(\mathbb{T}_{\Cl})=L.
	\end{eqnarray*}
This completes the proof of the lemma.
\end{proof}
\begin{proof}[Proof of Theorem \ref{thm: thmA}]
Results of Chekanov--Schlenk \cite[Lemmata~5.2 and 5.3]{ChekanovSchlenk} say that for any $x\!\in\! \mathbb{T}_{\Ch}$ the number of elements in $\ev^{-1}(x)\!\subset\! \mathcal{M}_1(\mathbb{T}_{\Ch},J_0)$ is five and they are regular with respect to the standard almost complex structure~$J_0$. Thus we have $n(\mathbb{T}_{\Ch})=1\in \Z_2$. By the main theorem, the Chekanov torus $\mathbb{T}_{\Ch}$ in $S^2\times S^2$ is not real. On the other hand, it follows from Lemma~\ref{lem: smoothinvolution} that $\mathbb{T}_{\Ch}$ is the fixed point set of a smooth involution.
\end{proof}

\section{Appendix}\label{sec: appendix}
\subsection{Some properties of $J$-holomorphic discs} 
Throughout this appendix, we assume the following setup. 
Let $(M,\ow)$ be a symplectic manifold and let $L$ be a Lagrangian in $M$. Fix $J\in \mathcal{J}$. We consider two $J$-holomorphic discs $u_j\colon (D^2,\p D^2)\to (M,L)$ for $j=1,2$.

The following lemma is a typical application of the (relative) Carleman similarity principle. See \cite[Proposition~2.1]{ZehAnnulus} or  \cite[Theorem~3.5]{LazGAFA}.
\begin{lemma}\label{lem: carleman}
Let $u\colon U\to M$ be a nonconstant $J$-holomorphic disc defined on a connected open set $U\subset D^2$. Then the following sets are finite:
\begin{itemize}
	\item The set of critical points $\crit(u):=\{z\in U \mid du(z)=0\}$ of $u$,
	\item the preimage $u^{-1}(x)$ for each $x\in M$.
\end{itemize}
\end{lemma}

The following proposition is proved in \cite[Theorem~4.13]{Laz}. See also \cite[Corollary~2.5.4]{McSalJholo} for the case of $J$-holomorphic spheres.
\begin{proposition}\label{prop: imagesamerep}
Suppose that $u_1$ and $u_2$ are simple. Then $u_1(D^2)=u_2(D^2)$ and $u_1(\p D^2)=u_2(\p D^2)$ if and only if there exists $\sigma\in \Aut(D^2)$ such that $u_1 = u_2 \circ \sigma$.
\end{proposition}
The following proposition can be regarded as the global version of the (relative) Carleman similarity principle. For the reader's convenience, we reproduce \cite[Proposition~5.1.3]{Hofer} here.
\begin{proposition}\label{prop: hofer}
	Suppose that $u_1(D^2)\ne u_2(D^2)$. If there exist sequences $z_\nu^1,z_\nu^2 \in D^2$ such that
	\begin{itemize}
		\item $z_\nu^j \to z^j_0\in D^2$ for $j=1,2$,
		\item $z_\nu^j\ne z_0^j$ for all $\nu \in \N$ and $j=1,2$,
		\item $u_1(z_\nu^1)=u_2(z_\nu^2)$ for all $\nu\in \N$ (and hence $u_1(z_0^1)=u_2(z_0^2)$),
	\end{itemize}
	that is, $u_1(z_0^1)=u_2(z_0^2)$ is a limit point of intersections $u_1(D^2)\cap u_2(D^2)$, then $z_0^j\in \crit(u_j)$ for $j=1,2$.
\end{proposition}
Motivated by the notation in \cite[Lemma~5.12]{SeidelQuiv}, we abbreviate by
\begin{equation}\label{eq: Su_1u_2set}
S(u_1,u_2):=\{z\in D^2 \mid u_1(z)\notin u_2(D^2)\}=\big(u_1^{-1}(u_2(D^2))\big)^c.	
\end{equation}
It is clear that $S(u_1,u_2)$ is open. For the following, see also \cite[Proposition~2.4.4]{McSalJholo} in the case of spheres.
\begin{proposition}\label{prop: imagediffdiscrete}
Suppose that $u_1$ and $u_2$ are simple and $u_1(D^2)\ne u_2(D^2)$. Then $S(u_1,u_2)$ is open and dense in $D^2$.
\end{proposition}
\begin{proof}
It suffices to show that $S(u_1,u_2)$ is dense. Arguing by contradiction we assume that $S(u_1,u_2)$ is not dense. Then there exists an open set $U$ in $D^2$ such that $U\cap S(u_1,u_2)=\emptyset$. Since $\crit(u_1)$ is finite, we can pick a point $z_0\in U\setminus \crit(u_1)$. Choose a sequence $z_\nu\in U\setminus \{z_0\}$ that converges to~$z_0$. Since we~know
$$
u_1(U)\subset u_2(D^2)
$$
and hence $u_1(z_\nu)\in u_2(D^2)$, we can choose a sequence $w_\nu \in D^2$ such that 
\begin{equation}\label{eq: denseseteq}
u_1(z_\nu)=u_2(w_\nu)\quad\text{ for all $\nu\in \N$}
\end{equation}
and $w_\nu \to w_0\in D^2$. By a limit argument, we obtain $u_1(z_0)=u_2(w_0)$. Recall that $z_0$ is not a critical point of $u_1$. By Proposition \ref{prop: hofer}, we must have that $w_\nu=w_0$ for all large $\nu\gg 1$. (Otherwise, $z_0\in \crit(u_1)$.) From \eqref{eq: denseseteq}, we see that
$$
z_\nu \in u_1^{-1}(u_2(w_0))
$$
for all large $\nu\gg1$. Hence, the set $u_1^{-1}(u_2(w_0))$ consists of infinitely many points, which contradicts Lemma \ref{lem: carleman}. This completes the proof.
\end{proof}
\subsection{A simple proof of the Main Theorem for $S^2\times S^2$}
In the particular case of a monotone $S^2\times S^2$, we give a short proof of the main theorem which contains simple but interesting ideas. 

Let $L=\Fix(R)$ be a real Lagrangian with $N_L= 2$ in $S^2\times S^2$, where $R$ is an antisymplectic involution of $S^2\times S^2$. Fix $J\in \mathcal{J}_R$. We define the \emph{moduli space of Maslov index 2 $J$-holomorphic discs with boundary on $L$ passing through $x\in L$},
$$
\mathcal{M}(L,J;x)=\{u\in \widehat{\mathcal{M}}(L,J)\mid u(1)=x\}/\Aut(D^2,1),
$$
where $\widehat{\mathcal{M}}(L,J)$ is defined in \eqref{eq: spaceofJdiscs}. 
The following observation is simple, but not obvious.
\begin{lemma}\label{lem: embedded_disc}
Let $J\in \mathcal{J}_R$. Then any Maslov index 2 $J$-holomorphic disc in $S^2\times S^2$ with boundary on $L$ is embedded.
\end{lemma}
\begin{proof}
	Let $u$ be such a disc in $S^2\times S^2$. Note that $S^2\cong \C P^1$ is the union of two  closed discs $D_1$ and $D_2$ such that $D_1$ is identified with the unit disc $D^2\subset \C$ and $D_2=\rho(D_1)$, where $\rho(z)=\bar{z}^{-1}$ is the antiholomorphic involution of $\C P^1=\C \cup \{\infty\}$. We define its \emph{double} $u^\sharp$ as the $J$-holomorphic sphere
	$$
	u^\sharp\colon S^2 \to S^2\times S^2, \quad u^\sharp(z)=\begin{cases}
		u(z), & z\in D_1\cong D^2, \\
		R(u(\rho(z))), & z\in D_2,
	\end{cases}
	$$
	which is \emph{smooth} by the Schwarz reflection principle.
	Recall that $H_2(S^2\times S^2;\Z)$ is generated by the homology classes $A_1=[S^2\times \{\pt\} ]$ and $A_2=[\{\pt\}\times S^2]$.
Since $c_1(S^2\times S^2)[u^\sharp]=\mu(u)=2$, we deduce that $[u^\sharp]=kA_1+(1-k)A_2$ for some $k\in \Z$.
Note that the minimal Chern number of $S^2\times S^2$ is 2 and hence that $u^\sharp$ is simple. It follows from the adjunction inequality, see \cite[Theorem~2.6.4]{McSalJholo} and \cite[Theorem~2.51]{Wendl},  that
\begin{equation}\label{eq: adj_ineq}
[u^\sharp]\bullet[u^\sharp] -c_1(S^2\times S^2)[u^\sharp]+\chi(S^2) = 2k(1-k)\ge 0,	
\end{equation}
which implies that $[u^\sharp]=A_i$ for some $i=1,2$. Moreover, since the equality holds in \eqref{eq: adj_ineq}, we deduce that $u^\sharp$ is embedded. As a result, the half disc $u$ is embedded as well.
\end{proof}
The automatic transversality result of Hofer--Lizan--Sikorav \cite[Theorem~2 and Remark~(1) on pg.\ 151]{HoferLizanSikorav} implies that any embedded $J$-holomorphic disc of Maslov index 2 is \emph{regular}, i.e., its associated operator ${\bf D}_u$ defined in Section~\ref{sec: eqtrans} is surjective. Hence, we obtain the following immediate consequence of Lemma~\ref{lem: embedded_disc}.
\begin{proposition}\label{prop: automatic_trans}
Every $J\in \mathcal{J}_R$ is regular.	
\end{proposition}
By Gromov's compactness theorem, $\mathcal{M}(L,J;x)$ is a compact zero dimensional smooth manifold for all $J\in \mathcal{J}_R$. The mod 2 cardinality of the finite set $\mathcal{M}(L,J;x)$ defines the invariant $n(L)\in \Z_2$, which is independent of the choice of $J\in \mathcal{J}_R$ and $x\in L$. To show that $n(L)=0$, consider the involution $\mathcal{R}$ of $\mathcal{M}(L,J;x)$ defined by
$$
\mathcal{R}[u]=[R\circ u\circ \rho],
$$
where $\rho$ is complex conjugation on $D^2\subset \C$.
\begin{proposition}\label{prop: nofixed_S2S2}
The involution $\mathcal{R}$ has no fixed point.
\end{proposition}
\begin{proof}
Assume to the contrary that $[u]\in \mathcal{M}(L,J;x)$ is a fixed point of $\mathcal{R}$. Then the image of $u$ is $R$-invariant, i.e., $\im u=R(\im u)$. Since $u$ is embedded by Lemma \ref{lem: embedded_disc}, we can define the antiholomorphic diffeomorphism of $D^2$,
	$$
	\psi:=u^{-1}\circ R\circ u.
	$$
Using the fact that $u(z)\in L=\Fix(R)$ for all $z\in \p D^2$, we see that $\psi|_{\p D^2}=\id$. Since there is no antiholomorphic diffeomorphism on $D^2$ which is the identity on the boundary, this is a contradiction.
\end{proof}
As a result, we obtain a simplified proof of the main theorem for $S^2\times S^2$, which is enough to prove Theorem \ref{thm: thmA}.

\subsection{The doubling construction}\label{sec: app_doubling}
Inspired by the doubling construction of Hofer--Lizan--Sikorav, we explain an alternative proof of equivariant transversality addressed in Section \ref{sec: eqtrans}.

Let $(M,\ow)$ be a closed spherically monotone symplectic manifold and let $L=\Fix(R)$ be a  real Lagrangian with $N_L= 2$ in $M$. Fix $J\in \mathcal{J}_R^\ell\subset \mathcal{J}^\ell$ of class~$C^\ell$ for a large $\ell\ge 1$. A \emph{real rational curve} $u'\colon S^2\to M$ is a $J$-holomorphic sphere $u'$ satisfying $R\circ u'\circ \rho=u'$, where $\rho(z)=\bar{z}^{-1}$ is the antiholomorphic involution on $S^2\cong \C P^1$. For a $J$-holomorphic disc $u\colon  (D^2,\p D^2)\to (M,L)$ of Maslov index 2 we write $u^\sharp\colon S^2\to M$ for its double, which is a real rational curve of Chern number 2.
Note that the minimal Chern number of $M$ is 2 as $N_L=2$. 
Since $J$-holomorphic spheres of Chern number 2 are necessarily simple, the corresponding \emph{universal moduli space},
$$
\mathcal{M}(\mathcal{J}^\ell)=\{(u'\colon S^2\to M,J)\mid \bar{\p}_J(u')=0,\ c_1(u')=2,\ J\in \mathcal{J}
^\ell\}
$$
is a Banach manifold, see  \cite[Proposition~3.2.1]{McSalJholo}. Consider the involution $\mathcal{R}'$ of $\mathcal{M}(\mathcal{J}^\ell)$ defined by
$$
\mathcal{R}'(u',J):=(R\circ u'\circ \rho,-R^*J).
$$
Its fixed point set 
$$
\Fix(\mathcal{R}')=\{(u',J)\in \mathcal{M}(\mathcal{J}^\ell)\mid R\circ u'\circ \rho=u',\ J\in \mathcal{J}_R^\ell\}
$$
defines the \emph{universal moduli space of real rational curves of Chern number 2}, which is a Banach submanifold of $\mathcal{M}(\mathcal{J}^\ell)$, see \cite[Proposition~1.9]{Welsch}. Hence, we can consider the projection
$$
\pi'\colon \Fix(\mathcal{R}')\to \mathcal{J}_R^\ell,\quad \pi'(u',J)=J,
$$
whose linearization $d\pi'$ at $(u',J)$ is a Fredholm operator with the same index as~${\bf D}_{u'}$. Here the operator ${\bf D}_{u'}\colon W^{1,p}(u'^*TM)\to L^p(\Lambda^{0,1}\otimes_J u'^*TM)$ is given by the linearization of the holomorphic curve equation for the case of spheres. By the Sard--Smale theorem, the set of regular values of $\pi'$ is a Baire subset~$\mathcal{J}_R^{\ell,\reg}$ of $\mathcal{J}_R^\ell$. 

Now suppose that $u$ is a Maslov index 2 $J$-holomorphic disc with boundary on $L$. Then the doubling construction of Hofer--Lizan--Sikorav \cite[Section~4 and Proposition~on pg.\ 158]{HoferLizanSikorav} says that if the operator
$$
{\bf D}_{u^\sharp}\colon W^{1,p}((u^\sharp)^*TM)\to L^p(\Lambda^{0,1}\otimes_J(u^\sharp)^*TM)
$$
is surjective, then so is the operator corresponding to the disc $u$,
$$
{\bf D}_u\colon W^{1,p}(u^*TM,u|_{\p D^2}^*TL)\to L^p(\Lambda^{0,1}\otimes_J u^*TM).
$$
This reproves Theorem \ref{thm: equitrans} with respect to the $C^\ell$-topology, and Taubes' trick completes the proof.

\subsection*{Acknowledgement}
The author cordially thanks Jo\'{e} Brendel, Hakho Choi, Urs Frauenfelder, Grigory Mikhalkin, and Felix Schlenk for fruitful discussions and suggestions. The author also thanks the anonymous referee whose kind suggestions have much improved this paper. This paper has grown up when the author stayed at the Institut de Math\'{e}matiques at Neuch\^{a}tel in February 2019. The author is grateful for its warm hospitality. This work is supported by Samsung Science and Technology Foundation under Project Number SSTF-BA1901-01.

\noindent
\small{Korea Institute for Advanced Study, 85 Hoegiro, Dongdaemun-gu, Seoul 02455, Republic of Korea\vspace{0.1cm}\\
\emph{E-mail address: }\texttt{joontae@kias.re.kr}
}

\end{document}